\documentclass[a4paper,11pt,reqno,english]{smfart}

\usepackage{aeguill}
\usepackage{enumerate}
\usepackage{amssymb,amsmath,latexsym,amsthm,amscd}
\usepackage[T1]{fontenc}
\usepackage{geometry}
\usepackage{url}
\usepackage[frenchb, main=english]{babel}
\usepackage[utf8]{inputenc}
\usepackage{mathrsfs}
\usepackage{xcolor}
\usepackage{comment}
\definecolor{violet}{rgb}{0.0,0.2,0.7}
\definecolor{rouge2}{rgb}{0.8,0.0,0.2}
\usepackage{tikz}
\usepackage{empheq}
\usepackage{tikz-cd}
\usetikzlibrary{matrix,arrows,decorations.pathmorphing}
\usepackage{hyperref}
\usepackage{mathpazo}
\hypersetup{
	bookmarks=true,         
	unicode=false,          
	pdftoolbar=true,        
	pdfmenubar=true,        
	pdffitwindow=false,     
	pdfstartview={FitH},    
	pdftitle={},    
	pdfauthor={},     
	colorlinks=true,       
	linkcolor=rouge2,          
	citecolor=violet,        
	filecolor=black,      
	urlcolor=cyan}           
\usepackage{enumitem}
\usepackage{appendix}

\usepackage{mathtools}
\usepackage{dsfont}

\theoremstyle{definition}
\newtheorem{defn}{Definition}[section]

\newtheorem*{ques}{Question}

\newtheorem*{ackn}{Acknowledgement}
\theoremstyle{remark}
\newtheorem{rmk}{Remark}[section]
\newtheorem*{rmk*}{Remark}

\theoremstyle{plain}
\newtheorem{thm}{Theorem}[section]

\newtheorem{lem}[thm]{Lemma}
\newtheorem{prop}[thm]{Proposition}
\newtheorem{cor}[thm]{Corollary}

\theoremstyle{plain}
\newtheorem{bigthm}{Theorem}

\newtheorem{bigset}[bigthm]{Setup}
\renewcommand{\thebigthm}{\Alph{bigthm}} 
\newtheorem*{bigrmk*}{Remark}
\makeatletter
\newcommand{\settheoremtag}[1]{
	\let\oldthebigthm\thebigthm
	\renewcommand{\thebigthm}{#1}
	\g@addto@macro\endbigset{
		\addtocounter{bigthm}{-1}
		\global\let\thebigthm\oldthebigthm}
}
\makeatother

\numberwithin{equation}{section}
\renewcommand{\labelenumi}{(\roman{enumi})}
\newlist{steps}{enumerate}{1}
\setlist[steps, 1]{label = Step \arabic*:}
\DeclareFontFamily{U}{MnSymbolC}{}
\DeclareSymbolFont{MnSyC}{U}{MnSymbolC}{m}{n}
\DeclareFontShape{U}{MnSymbolC}{m}{n}{
	<-6>  MnSymbolC5
	<6-7>  MnSymbolC6
	<7-8>  MnSymbolC7
	<8-9>  MnSymbolC8
	<9-10> MnSymbolC9
	<10-12> MnSymbolC10
	<12->   MnSymbolC12}{}
\DeclareMathSymbol{\intprod}{\mathbin}{MnSyC}{'270}

\DeclareMathOperator{\tr}{tr}
\DeclareMathOperator{\pr}{pr}
\DeclareMathOperator{\Vol}{Vol}
\DeclareMathOperator{\Ric}{Ric}

\DeclareMathOperator{\supp}{supp}

\def\1{\mathds{1}}

\newcommand\oh{\frac{1}{2}}

\newcommand{\ii}{\mathrm{i}}

\newcommand{\loc}{\mathrm{loc}}

\newcommand\dt{\delta}

\newcommand\vep{\varepsilon}

\newcommand\om{\omega}

\newcommand\af{\alpha}
\newcommand\bt{\beta}

\newcommand\Dt{\Delta}

\newcommand\Gm{\Gamma}

\newcommand\RG{\mathrm{G}}

\newcommand\BN{\mathbb{N}}
\newcommand\BZ{\mathbb{Z}}
\newcommand\BR{\mathbb{R}}
\newcommand\BC{\mathbb{C}}

\newcommand\BS{\mathbb{S}}
\newcommand\BD{\mathbb{D}}
\newcommand\BCP{\mathbb{CP}}

\newcommand\CC{\mathcal{C}}
\def\CD{\mathcal{D}}

\newcommand\CM{\mathcal{M}}
\newcommand\CO{\mathcal{O}}

\newcommand\CU{\mathcal{U}}

\newcommand\CX{\mathcal{X}}

\newcommand\lt{\left}
\newcommand\rt{\right}

\newcommand\ra{\rightarrow}

\newcommand\pl{\partial}
\newcommand\db{\bar{\partial}}

\newcommand\ddb{\partial \bar{\partial}}

\newcommand\dd{\mathrm{d}}
\newcommand\dc{\mathrm{d}^{\mathrm{c}}}
\newcommand\ddc{\mathrm{dd}^{\mathrm{c}}}

\newcommand\norm[1]{\left\lVert {#1} \right\rVert}

\newcommand\abs[1]{\left\lvert {#1} \right\rvert}

\newcommand\w{\wedge}

\newcommand\reg{\mathrm{reg}}
\newcommand\sing{\mathrm{sing}}

\newcommand\set[2]{\left\{ {#1} \, \middle| \, {#2} \right\}}
\newcommand\iprod[2]{\left\langle {#1}, {#2} \right\rangle}

\newcommand\res[2]{\left. {#1} \right|_{#2}} 

\newcommand{\RN}[1]{\textup{\uppercase\expandafter{\romannumeral#1}}}

\title{Singular Gauduchon metrics}

\date{\today}

\author{Chung-Ming Pan}

\address{
Institut de Math\'{e}matiques de Toulouse; UMR 5219, Universit\'{e} de Toulouse; CNRS, UPS, 118 route de Narbonne, F-31062 Toulouse Cedex 9, France	
}

\email{Chung\_Ming.Pan@math.univ-toulouse.fr}

\begin{document} 



\maketitle

\begin{abstract}
In 1977, Gauduchon proved that on every compact hermitian manifold $(X, \om)$ there exists a conformally equivalent hermitian metric $\om_\RG$ which satisfies $\ddc \om_\RG^{n-1} = 0$. 
In this note, we extend this result to irreducible compact singular hermitian varieties which admit a smoothing.
\end{abstract}

\tableofcontents

\section*{Introduction}
Let $X$ be an $n$-dimensional compact complex manifold equipped with a positive definite smooth $(1,1)$-form $\om$.
We also call $\om$ a hermitian metric because such $\om$ corresponds to a hermitian metric.
A famous theorem of Gauduchon \cite{Gauduchon1977} says that there exists a metric $\om_\RG$ in the conformal class of $\om$ such that $\ddc \om_\RG^{n-1} = 0$ and the metric $\om_\RG$ is unique up to a positive multiple.  
These kind of metrics are since then called \textit{Gauduchon metrics}. 
The conformal factor $\rho$ satisfying $\om_\RG^{n-1} = \rho \om^{n-1}$ is called \textit{Gauduchon factor}.

\smallskip

In complex geometry, finding canonical metrics on complex manifolds is a central problem. 
Two celebrated examples are Yau's solution of the Calabi conjecture \cite{Yau1978} and Uhlenbeck--Yau's characterization of the the existence of hermitian-Einstein metrics on stable vector bundles \cite{UY1986}.
These theorems are established on K\"ahler manifolds.
When the manifold is non-K\"ahler, the analysis is more difficult because hermitian metrics are no longer closed.
In such cases, Gauduchon metrics provide a useful substitute. 
For instance, Li--Yau~\cite{LiYau1987} used Gauduchon metrics to define the slope stability of vector bundles on compact non-K\"{a}hler manifolds. As a consequence, they generalized Uhlenbeck--Yau's theorem to non-K\"ahler setting. 
For generalized Calabi-Yau type problem in non-K\"ahler context, Tosatti--Weinkove~\cite{TW2010} showed that for arbitrary representative $\Psi \in c_1^{\mathrm{BC}}(X)$ of the first Bott--Chern class of $X$, there exists a hermitian metric $\om$ such that $\Ric(\om) = \Psi$ by solving complex Monge--Amp\`ere equations. 
In their proof, Gauduchon metrics play an important role to simplify calculations.
Furthermore, Sz\'{e}kelyhidi--Tosatti--Weinkove~\cite{STW2017} proved that one can even find a Gauduchon metric with prescribed Chern--Ricci curvature. 
On the other hand, Angella--Calamai--Spotti~\cite{ACS2017} studied the Chern--Yamabe problem (i.e. finding constant Chern scalar curvature metrics in the conformal class of a given metric $\om$). 
They used Gauduchon metrics to define a conformal invariant called the Gauduchon degree and showed that if a metric $\om$ has non-positive Gauduchon degree then the Chern--Yamabe problem admits a solution.
For more applications and results about Gauduchon metrics, the interested reader is referred to \cite{FinoUgarte2013, FWW2013, Li2021Gauduchon} and the references therein.

\smallskip

From an algebraic point of view, singularities are ubiquitous as they occur in various contexts, notably in the Minimal Model Program and moduli theory. 
Ueno~\cite{Ueno1975} found a birational class of $3$-dimensional complex manifolds which does not admit a smooth minimal model.
In moduli theory, it is necessary to deal with singular varieties when compactifying moduli spaces of smooth manifolds.
Already in dimension one, the fundamental domain of moduli space of elliptic curves is non-compact and nodal curves lie on its boundary.
On the other hand, in non-K\"ahler geometry, investigating singular varieties admitting a non-K\"ahler smoothing is an essential issue due to close interactions of string theory and mathematics established over the past 40-years. 
In 1980s, a large class of non-K\"ahler Calabi--Yau threefolds was built via conifold transitions introduced by Clemens~\cite{Clemens1983} and Friedman~\cite{Friedman1986}. 
Roughly speaking, the process goes as follows: contracting a collection of disjoint $(-1,-1)$-curves from a K\"ahler Calabi--Yau threefold $X$ to get a singular Calabi--Yau variety $X_0$ and then smoothing singularities of $X_0$, one obtains a family of Calabi--Yau threefolds $(X_t)_{t \neq 0}$ which are generally non-K\"ahler.
Thus, it is important to understand the geometric structure on $X_t$ induced by the original Calabi--Yau manifold $X$.
Experts believe that the Hull--Strominger system \cite{Hull1986,Strominger1986} provides a natural candidate.
These models attracted a lot of attentions in recent years (cf. \cite{Reid1987,Friedman1991,Tian1992,Rossi2006,FY2008,Chuan2012,FLY2012,PPZ2018,CPY2021} and the references therein).

\smallskip

Given these considerations, it is legitimate to look for canonical metrics or special metrics on singular hermitian varieties.
In this note, we shall focus on Gauduchon metrics.
A standard way to give a metric structure on a singular complex space is to restrict an ambient metric in local embeddings (see Definition~\ref{defn:metrics} for the precise definition). 
Then we address the following question:
\begin{ques}
Suppose that $X$ is an irreducible, reduced, compact complex space equipped with a hermitian metric $\om$.
Can one find a Gauduchon metric $\om_\RG$ in the conformal class of $\om$?
\end{ques}

The purpose of this note is to give partial answers in the setup of smoothable singularities.
This means that the variety can be approximated by a family of smooth manifolds and the hermitian metric is the restriction of an ambient smooth metric. 
Here are the precise statements:
\settheoremtag{(S)}
\begin{bigset}\label{set:family setting}
Let $\CX$ be an $(n+1)$-dimensional, irreducible, reduced, complex analytic space.
Suppose that 
\begin{itemize}
    \item $\pi: \CX \ra \BD$ is a proper, surjective, holomorphic map with connected fibres;
    \item $\pi$ is smooth over the punctured disk $\BD^\ast$;
    \item the central fibre $X_0$ is an $n$-dimensional, irreducible, reduced, compact complex analytic space.
\end{itemize}
Let $\om$ be a hermitian metric on $\CX$ in the sense of Definition~\ref{defn:metrics}.
For each $t \in \BD$, we define the hermitian metric $\om_t$ on fibre $X_t$ by restriction (i.e. $\om_t := \res{\om}{X_t}$).
\end{bigset}

In the above Setup~\ref{set:family setting}, on each smooth fibre $X_t$, there exists a Gauduchon factor $\rho_t$ with respect to $\om_t$ by Gauduchon's theorem.
We may normalize these Gauduchon factors such that $\inf_{X_t} \rho_t =1$.
Then we show the existence of a smooth Gauduchon factor on the smooth part of the central fibre.
\begin{bigthm}[cf. Corollary~\ref{cor:G bdd family} and Theorem~\ref{thm:Gau current}]\label{bigthm:Gauduchon exist family}
In Setup~\ref{set:family setting}, we have the following properties:
\begin{enumerate}
    \item There is a uniform constant $C_G \geq 1$ such that the normalized Gauduchon factors $\rho_t$ are bounded between $1$ and $C_G$ on each smooth fibre $X_t$ for all $t \in \BD^\ast_{1/2}$;
    \item There exists a smooth bounded Gauduchon factor $\rho$ of $\om_0$ on $X_0^\reg$. 
\end{enumerate}
\end{bigthm}

Thus, $X_0$ admits a \textit{bounded Gauduchon metric} (see Definition~\ref{defn:Gauduchon metrics}) in the conformal class of $\om_0$.
The idea of proof is to approximate the Gauduchon factor on the singular fibre by the normalized Gauduchon factors on nearby smooth fibres.

\smallskip

Next, we assume that $X$ is a variety endowed with a bounded Gauduchon metric $\om_\RG$.
We show that the trivial extension of the $(n-1,n-1)$-form $\om_\RG^{n-1}$ through $X^\sing$ is a pluriclosed current on $X$.
Moreover, we also prove the analogous uniqueness result of Gauduchon.

\begin{bigthm}\label{bigthm:Gauduchon extention}
Suppose that $X$ is an $n$-dimensional, irreducible, reduced, compact complex space endowed with a bounded Gauduchon metric $\om_\RG$.
Then we have the following uniqueness and extension properties:
\begin{enumerate}
    \item If $\om_\RG'$ is another bounded Gauduchon metric in the conformal class of $\om_\RG$, then $\om_\RG'$ must be a positive multiple of $\om_\RG$;
    \item Let $T$ be the positive $(n-1,n-1)$-current obtained as the trivial extension of $\om_\RG^{n-1}$.
    Then $T$ is a pluriclosed current.
\end{enumerate}
\end{bigthm}

The main strategy is to use "good" cutoff functions. 
Complement of proper analytic subsets (eg. $X^\reg = X \setminus X^\sing$) admit exhaustion functions with small $L^2$-gradient.
This enables us to show that the trivial extension of $\om_\RG^{n-1}$ as a current on $X$ satisfies $\ddc T = 0$ in the sense of currents globally on $X$. 
The uniqueness property follows similarly.

\smallskip

This note is organized as follows:
Section~\ref{sec:prelim} provides some backgrounds.
Section~\ref{sec:Gauduchon regular} contains sup-estimate of normalized Gauduchon factors in families (the first property in Theorem~\ref{bigthm:Gauduchon exist family}).
In Section~\ref{sec:Gauduchon singular}, we show the existence of bounded Gauduchon factors on the central fibre (the second part of Theorem~\ref{bigthm:Gauduchon exist family}) and give the proof of Theorem~\ref{bigthm:Gauduchon extention}.

\begin{ackn}
The author is grateful to his thesis advisors Vincent~Guedj and Henri~Guenancia for their supports, suggestions and encouragements.
The author thanks Tat~Dat~T\^{o} for pointing out the reference \cite{ACS2017} and Tsung-Ju Lee for helpful comments.
The author would like to thank the anonymous referees for useful comments and suggestions.
This work has benefited from State aid managed by the ANR under the ”PIA” program bearing the reference ANR-11-LABX-0040 (research project HERMETIC). 
The author is also partially supported by the ANR project PARAPLUI and the EUR MINT project ANR-18-EURE-0023.
\end{ackn}

\section{Preliminaries}\label{sec:prelim}
In this section, we recall some notations and conventions which will be used in the sequel.
We define the twisted exterior derivative by
\[
	\dc = \frac{\ii}{2}(\db - \pl)
\]
and we then have $\ddc = \ii\ddb$.
We say that a form is \textit{pluriclosed} if it is $\ddc$-closed.
We denote by
\begin{itemize}
    \item $\BD_r := \set{z \in \BC}{|z|<r}$ the open disk of radius $r$;
    \item $\BD^\ast_{r} := \set{z \in \BC}{0 < |z| < r}$ the punctured disk of radius $r$.
\end{itemize}
When $r=1$, we simply write $\BD := \BD_1$ and $\BD^\ast := \BD^\ast_1$.

\subsection{Metrics on singular spaces} 
Let $X$ be a reduced complex analytic space of pure dimension $n \geq 1$.
We will denote by $X^\reg$ the complex manifold of regular points of $X$ and $X^\sing := X \setminus X^\reg$ the singular set of $X$.
Now, we give the definition of hermitian metrics on reduced complex analytic space $X$:

\smallskip

\begin{defn}\label{defn:metrics}
A hermitian metric $\om$ on $X$ is the data of a hermitian metric $\om$ on $X^\reg$ such that given any local embedding $X \xhookrightarrow[\loc.]{} \BC^N$, $\om$ extends smoothly to a hermitian metric on $\BC^N$.
\end{defn}

\begin{rmk}
The notion of smooth forms (and hermitian metrics) as above does not depend on the choice of local embeddings (see \cite[page 14]{Demailly1985}).
Hermitian metrics always exist: one can use local embeddings and then glue local data of hermitian metrics by a partition of unity.
\end{rmk}

\smallskip

Note that in Definition~\ref{defn:metrics} a hermitian metric on $X$ is more than just a metric on $X^\reg$.
Now, we define Gauduchon metrics in the following two different concepts:

\smallskip

\begin{defn}\label{defn:Gauduchon metrics}
We say that a hermitian metric $\om_\RG$ on $X^\reg$ is
\begin{enumerate}
    \item \textit{Gauduchon} if it satisfies $\ddc \om^{n-1}_\RG = 0$ on $X^\reg$;
    \item \textit{bounded Gauduchon metric} on $X$ if there exist a hermitian metric $\om$ and a positive bounded smooth function $\rho$ defined on $X^\reg$ such that $\om_\RG = \rho^{\frac{1}{n-1}} \om$ and $\ddc \om^{n-1}_\RG = 0$ on $X^\reg$.
\end{enumerate}
\end{defn}

\smallskip

We define the complex Laplacian and the norm of gradients with respect to $\om$ by
\begin{align*}
    \Dt_\om f &:= \tr_\om (\ddc f) = \frac{n \ddc f \w \om^{n-1}}{\om^n}\\
    \abs{\dd f}_\om^2 &:= \tr_\om(\dd f \w \dc f) = \frac{n \dd f \w \dc f \w \om^{n-1}}{\om^n}.
\end{align*}

\subsection{Currents on singular spaces}
Recall that smooth forms on $X$ are defined as restriction of smooth forms in local embeddings.
We denote by 
\begin{itemize}
    \item $\CD_{p,q}(X)$ the space of compactly supported smooth forms of bidegree $(p,q)$;
    \item $\CD_{p,p}(X)_\BR$ the space of real smooth $(p,p)$-forms with compact support.
\end{itemize} 
The notion of currents on $X$, $\CD'_{p,q}(X)$ and $\CD'_{p,p}(X)_\BR$, is defined by acting on compactly supported smooth forms on $X$. 
The operators $\dd$, $\dc$ and $\ddc$ are well-defined by duality (see \cite{Demailly1985} for detail arguments).

\subsection{Example}
We give an example of a non-K\"ahler variety satisfying Setup~\ref{set:family setting} extracted from \cite{Lu_Tian_1994}. 
The manifold $M = (\Gm_1 \times \Gm_2) \cap H \subset \BCP^3 \times \BCP^3$ is a simply connected Calabi--Yau threefold with $b_2 = 14$,
where 
\begin{align*}
	\Gm_1 &= \{x_0^3 + x_1^3 + x_2^3 + x_3^3 = 0\} \subset \BCP^3,\\
	\Gm_2 &= \{y_0^3 + y_1^3 + y_2^3 + y_3^3 = 0\} \subset \BCP^3,\\
	H &= \{x_0 y_0 + x_1 y_1 + x_2 y_2 + x_3 y_3 = 0\} \subset \BCP^3 \times \BCP^3.
\end{align*}
There exists $15$ disjoint rational curves $\ell_1, \cdots, \ell_{15}$ with normal bundles $\CO_{\BCP^1}(-1)^{\oplus 2}$, $\ell_1, \cdots, \ell_{14}$ generate $H_2(M,\BZ)$, and there exists $a_j \neq 0$ for $j \in \{1,\cdots,15\}$ such that $\sum_{j=1}^{15} a_j [\ell_j] = 0$.
Then one can contract these curves and get a singular space $X_0$ with $15$ ordinary double points. 
Moreover, there is a smoothing $\pi: \CX \ra \BD$ of $X_0$ such that for all $t\neq 0$, and $X_t$ is diffeomorphic to a connected sum of $\BS^3 \times \BS^3$. 
This implies that $X_t$ does not admit a K\"ahler metric since $b_2(X_t) = 0$.
One can check that $X_0$ is not K\"ahler.

\subsection{Remarks on the family setting}
Under Setup~\ref{set:family setting}, the family of metrics $(\om_t)_{t}$ satisfies the following properties:
For all $t \in \overline{\BD}_{1/2}$, there is a constant $B \geq 0$ independent of $t$ such that
\begin{equation}\label{eq:family constant B}
    -B \om_t^n \leq \ddc_t \om_t^{n-1} \leq B \om_t^n
\end{equation}
where $\dc_t$ is the twisted exterior derivative with respect to the complex structure of $X_t$.
Indeed, in a local embedding, we have $-B\om^n \leq \ddc \om^{n-1} \leq B \om^n$ on $\CX$; hence (\ref{eq:family constant B}) is just the restriction on each fibre $X_t$.
On the other hand, the volume of $(X_t, \om_t)$ is comparable to the volume of $(X_0, \om_0)$ for all $t \in \overline{\BD}_{1/2}$.
Namely, we have a uniform constant $C_V \geq 1$ such that
\begin{equation}\label{eq:family constant C_V}
    C_V^{-1} \leq \Vol_{\om_t}(X_t) \leq C_V, \quad \forall t \in \overline{\BD}_{1/2}.
\end{equation}
The lower bound is obvious. 
One can prove the upper bound by the continuity of the total mass of currents $(\om^n \w [X_t])_{t \in \overline{\BD}_{1/2}}$. 
The proof goes as follows: the current of integration $[X_t]$ can be written as $\ddc \log|\pi - t|$ by the Poincar\'e--Lelong formula. 
Since $|\pi - t|$ converges to $|\pi|$ uniformly when $t \ra 0$, $\log|\pi - t|$ converges to $\log|\pi|$ almost everywhere and thus $\log|\pi - t| \ra \log|\pi|$ in $L^1$ when $t \ra 0$ by Hartogs' lemma. 
Therefore, $\om^n \w [X_t] \xrightarrow[t \ra 0]{} \om^n \w [X_0]$ in the sense of currents and this implies $\Vol_{\om_t}(X_t) \xrightarrow[t \ra 0]{} \Vol_{\om_0}(X_0)$. 
Thus, using the compactness of $X_0$, we obtain a uniform upper bound $C_V$ of $\Vol_{\om_t}(X_t)$ for all $t \in \overline{\BD}_{1/2}$.

\smallskip

Finally, we give a remark on non-smoothable singularities:
The first non-smoothable example was given by Thom and reproduced by Hartshorne~\cite{Hartshorne1974}. 
They considered a cone in $\BC^6$ over the Segre embedding of $\BCP^1 \times \BCP^2$ into $\BCP^5$ and they proved that the cone does not admit a smoothing because of a topological constraint.

\section{Gauduchon metrics on smooth fibres}\label{sec:Gauduchon regular}
The aim of this section is to prove the uniform boundedness of normalized Gauduchon factors $\rho_t$ with respect to $\om_t$ on smooth fibres $X_t$ in Setup~\ref{set:family setting}.
First of all, we fix a compact hermitian manifold $(X, \om)$.
Suppose that $B \geq 0$ is a constant such that $-B \om^n \leq \ddc\om^{n-1} \leq B \om^n$.
From Gauduchon's theorem \cite{Gauduchon1977}, there exists a unique positive smooth function $\rho \in \ker \Dt_\om^\ast$ or equivalently 
\begin{equation*}
    \ddc(\rho \om^{n-1}) = \ddc\rho \w \om^{n-1} + \dd\rho \w \dc\om^{n-1} - \dc\rho \w \dd\om^{n-1} + \rho\ddc\om^{n-1} = 0
\end{equation*}
such that $\inf_X \rho = 1$ and $\om_\RG = \rho^{\frac{1}{n-1}} \om$ is a Gauduchon metric.
Then we prove that the Gauduchon factor is bounded by geometric quantities.
\begin{thm}\label{thm:G bdd fix mfd}
Suppose that $(X,\om)$ is an $n$-dimensional compact hermitian manifold.
If $\rho \in \ker \Dt_\om^\ast$ and $\inf_X \rho = 1$, then there is a constant $C_G$ depending only on $n, B, C_S, C_P$ and $\Vol_\om(X)$ such that
\[
    \sup_X \rho \leq C_G,
\]
where $C_S$ and $C_P$ are Sobolev and Poincar\'{e} constants.
\end{thm}

The proof of Theorem~\ref{thm:G bdd fix mfd} is inspired by the paper of Tosatti--Weinkove~\cite{TW2010}.
We shall apply Moser's iteration twice to get an upper bound of $\rho$.
On the other hand, under Setup~\ref{set:family setting}, the Sobolev and Poincar\'{e} constants of the fibres $X_t$ are uniformly bounded independently of $t$.
The uniform Sobolev constant in family comes from Wirtinger inequality and Michael--Simon's Sobolev inequality on minimal submanifolds \cite{MS1973}. 
The study of Poincar\'{e} constant in families goes back to Yoshikawa \cite{Yoshikawa1997} and Ruan--Zhang~\cite{RZ2011}.
For convenience, the reader is also referred to \cite[Proposition 3.8 and 3.10]{DGG2020}. 
Although they only stated the properties on a family of K\"ahler spaces, the proof does not rely on K\"ahler structures.

\begin{prop}\label{prop:SPcst}
Suppose that $\pi: (\CX,\om) \ra \BD$ is a family of compact hermitian varieties in Setup~\ref{set:family setting}.
For all $t \in K \Subset \BD$, there exists uniform Sobolev and Poincar\'{e} constants $C_S(K)$ and $C_P(K)$ such that
\[
    \forall f \in L^2_1(X^\reg_t),\quad 
    \lt( \int_{X_t} \abs{f}^{\frac{2n}{n-1}} \om_t^n \rt)^{\frac{n-1}{n}}
    \leq C_S \lt( \int_{X_t} \abs{\dd f}^2_{\om_t} \om_t^n 
    + \int_{X_t} \abs{f}^2 \om_t^n \rt)
\]
and
\[
    \forall f \in L^2_1(X^\reg_t) \text{ and } \int_{X_t} f \om_t^n = 0,\quad
    \int_{X_t} \abs{f}^2 \om_t^n 
    \leq C_P \int_{X_t} \abs{\dd f}_{\om_t}^2 \om_t^n.
\]
\end{prop}

\begin{rmk}
The irreducible condition is crucial in the proof of uniform Poincar\'e inequality. 
Assume that $X_0$ has two irreducible components $X_0'$ and $X_0''$.
Consider a function $f$ on $X_0^\reg$ and it is defined by
\[
    \begin{cases}
        f = 1/\Vol_{\om_t}(X'_0) & \text{on } (X_0')^\reg\\
        f = -1/\Vol_{\om_t}(X''_0) & \text{on } (X_0'')^\reg\\
        f = 0 & \text{otherwise }
    \end{cases}.
\]
Then it is not hard to see that the RHS of Poincar\'e inequality is zero but the LHS is positive.
This yields a contradiction.
One can also construct a "quantitative" version of that example. Namely, the Poincar\'e constant $C_{P,t}$ on each smooth fibre $X_t$ blows-up when $t \ra 0$ (see eg. \cite{Yoshikawa1997, DGG2020}).
\end{rmk}

Combining Theorem~\ref{thm:G bdd fix mfd}, Proposition~\ref{prop:SPcst}, (\ref{eq:family constant B}) and (\ref{eq:family constant C_V}), we obtain the following uniform estimate in the family setting.

\begin{cor}\label{cor:G bdd family}
Let $\pi: (\CX,\om) \ra \BD$ be a family of compact hermitian manifolds as in Setup~\ref{set:family setting}. 
Then there exists a constant $C_G>0$ such that
\begin{equation*}
    \sup_{X_t} \rho_t \leq C_G, 
    \text{ and } \inf_{X_t} \rho_t = 1
    \quad \text{for all } t \in \BD_{1/2}^\ast
\end{equation*}
where $\rho_t$ is the Gauduchon factor with respect to $(X_t, \om_t)$.
\end{cor}

\subsection{Proof of Theorem~\ref{thm:G bdd fix mfd}}
In this subsection, we shall establish two gradient estimates ((\ref{eq:g estimate 1}) and (\ref{eq:g estimate 2})) and then apply Moser's iteration argument to obtain an upper bound of $\rho$.
In order to check the dependence on each data, we shall formulate the dependence of given constants.
For convenience, we write $V := \Vol_\om(X)$.
We start the proof with a useful formula: 

\begin{lem}\label{lem:Gau estimate}
Suppose that $\rho \in \ker \Dt^\ast_\om$. 
Then we have
\[
    \int_X F'(\rho) \dd\rho \w \dc\rho \w \om^{n-1}
    = \int_X G(\rho) \ddc\om^{n-1}
\]
where $F$ and $G$ are two real $\CC^1$-functions defined on $\BR_{>0}$ which satisfy $x F'(x) = G'(x)$. 
\end{lem}

\begin{proof}
Since $\rho \in \ker \Dt_\om^\ast$, we have $\ddc(\rho \om^{n-1}) = 0$ equivalently.
From Stokes' theorem, it follows that 
\begin{align*}
    0 &= \int_X F(\rho) \ddc(\rho \om^{n-1})
    = - \int_X F'(\rho)\dd\rho \w \dc(\rho \om^{n-1})\\
    &= - \int_X F'(\rho)\dd\rho \w \dc\rho \w \om^{n-1} 
    - \int_X F'(\rho)\rho \dd\rho \w \dc\om^{n-1}.
\end{align*}
By the assumption $xF'(x) = G'(x)$, we obtain the desired formula
\begin{align*}
    \int_X F'(\rho) \dd\rho \w \dc\rho \w \om^{n-1}
    &= -\int_X \rho F'(\rho)\dd\rho \w \dc\om^{n-1}\\
    &= -\int_X \dd G(\rho) \w \dc \om^{n-1}
    = \int_X G(\rho) \ddc \om^{n-1}.
\end{align*}
\end{proof}

Consider $F'(x) = \frac{np^2}{4} x^{-p-2}$ and $G(x) = -\frac{np}{4}x^{-p}$ for $p \geq 1$.
By Lemma~\ref{lem:Gau estimate}, we have the following inequality
\begin{equation}\label{eq:g estimate 1}
\begin{split}
    \int_X \abs{\dd\lt(\rho^{-\frac{p}{2}}\rt)}_\om^2 \om^n
    &= \int_X n \dd \lt(\rho^{-\frac{p}{2}}\rt) \w \dc \lt(\rho^{-\frac{p}{2}}\rt) \w \om^{n-1}\\
    &= \int_X F'(\rho) \dd\rho \w \dc\rho \w \om^{n-1}
    = \int_X G(\rho) \ddc\om^{n-1}\\
    &= \frac{np}{4} \int_X \rho^{-p} \lt(-\ddc\om^{n-1}\rt)\\
    &\leq p\frac{nB}{4} \int_X \rho^{-p} \om^n.
\end{split}
\end{equation}
Put $\bt = \frac{n}{n-1} > 1$.
Combining Sobolev inequality and (\ref{eq:g estimate 1}), we get
\begin{align*}
    \lt(\int_X (\rho^{-1})^{p\bt} \om^n\rt)^{\frac{1}{\bt}} 
    &\leq C_S \lt( \int_X \abs{\dd \lt(\rho^{-\frac{p}{2}}\rt)}^2_\om \om^n + \int_X \abs{\rho^{-\frac{p}{2}}}^{2} \om^n\rt)\\
    &\leq C_S \lt( p \frac{n B}{4} \int_X (\rho^{-1})^p \om^n + \int_X (\rho^{-1})^p \om^n\rt)\\
    &\leq p \lt( \frac{nB}{4} + 1 \rt) C_S \int_X \rho^{-p} \om^n.
\end{align*}
For all $p \geq 1$, we have
\begin{align*}
    \norm{\rho^{-1}}_{L^{p\bt}(X,\om)} 
    &\leq p^{\frac{1}{p}} C_1^{\frac{1}{p}} \norm{\rho^{-1}}_{L^{p}(X,\om)}
\end{align*}
where $C_1$ is a constant depending only on $n, B, C_S$.
Inductively, we obtain
\begin{align*}
    \norm{\rho^{-1}}_{L^{p\bt^k}(X,\om)}
    &\leq p^{\frac{1}{p} \sum_{j=0}^{k-1} \frac{1}{\bt^j}} 
    \bt^{\frac{1}{p} \sum_{j=0}^{k-1} \frac{j}{\bt^j}}
    C_1^{\frac{1}{p} \sum_{j=0}^{k-1} \frac{1}{\bt^j}}
    \norm{\rho^{-1}}_{L^{p}(X,\om)}\\
    &\leq p^{\frac{n}{p}} \bt^{\frac{n(n-1)}{p}} 
    C_1^{\frac{n}{p}}
    \norm{\rho^{-1}}_{L^{p}(X,\om)}.
\end{align*}
Let $p=1$ and $k \ra \infty$ and thus we get
\[
    1 = \sup_X \rho^{-1} 
    \leq \bt^{n(n-1)} C_1^n \int_X \rho^{-1} \om^n.
\]
Therefore, the $L^1$-norm of $\rho^{-1}$ is bounded away from zero by a constant $\dt$ which depends only on $n, B, C_S$:
\[
    \int_X \rho^{-1} \om^n \geq \frac{1}{\bt^{n(n-1)} C_1^n} =:2\dt.
\]
Choosing sufficiently small $\dt$, we may assume $A := V/\dt \geq 1$.
Then we have 
\[
    2\dt \leq \int_{\{\rho<A\}} \rho^{-1} \om^n + \int_{\{\rho \geq A\}} \rho^{-1} \om^n
    \leq \int_{\{\rho<A\}} \om^n + \frac{1}{A} \int_X \om^n.
\]
Hence, the volume of $\{\rho < A\}$ is bounded away from zero:
\begin{equation}\label{eq:dt-vol}
    \int_{\{\rho < A\}} \om^n \geq \dt.
\end{equation}

Now, we consider $F'(x) = \frac{n(p+1)^2}{4} \frac{(\log x)^{p-1}}{x^2}$ and $G(x) = \frac{n(p+1)^2}{4p} (\log x)^p$ for $p \geq 1$. 
From Lemma~\ref{lem:Gau estimate}, we find the following estimate
\begin{equation}\label{eq:g estimate 2}
\begin{split}
    \int_X \abs{\dd (\log \rho)^{\frac{p+1}{2}}}_\om^2 \om^n
    &= \frac{n(p+1)^2}{4} \int_X \frac{(\log \rho)^{p-1}}{\rho^2} \dd\rho \w \dc\rho \w \om^{n-1}\\
    &= \int_X F'(\rho) \dd\rho \w \dc\rho \w \om^{n-1}
    = \int_X G(\rho) \ddc\om^{n-1}\\
    &= \frac{n(p+1)^2}{4p} \int_X (\log \rho)^p \ddc \om^{n-1}\\
    &\leq p n B \int_X (\log\rho)^p \om^n.
\end{split}
\end{equation}
Again, Sobolev inequality, (\ref{eq:g estimate 2}), and H\"older inequality yield the following inequalities:
\begin{align*}
    \lt(\int_X (\log\rho)^{(p+1)\bt} \om^n \rt)^{\frac{1}{\bt}}
    &\leq C_S \lt( \int_X \abs{\dd (\log \rho)^{\frac{p+1}{2}}}^2_\om \om^n + \int_X (\log\rho)^{p+1} \om^n \rt)\\
    &\leq C_S\lt( p n B \int_X (\log\rho)^p \om^n + \int_X (\log\rho)^{p+1} \om^n\rt)\\
    &\leq C_S\lt( p n B V^{\frac{1}{p+1}} \lt(\int_X (\log\rho)^{p+1} \om^n\rt)^{\frac{p}{p+1}} + \int_X (\log\rho)^{p+1} \om^n \rt)\\
    &\leq (p+1) 2 n B C_S \max\{V,1\} \max\lt\{\int_X (\log\rho)^{p+1} \om^n,1 \rt\}.
\end{align*}
Write $q = p+1 \geq 2$. 
We have 
\[
    \norm{\log \rho}_{L^{q\bt}(X,\om)}
    \leq q^{\frac{1}{q}} C_2^{\frac{1}{q}} \max\lt\{\norm{\log\rho}_{L^q(X,\om)}, 1 \rt\}
\]
where $C_2 > 0$ is a constant depending only on $n, B, C_S, V$.
Using the similar strategy of Moser's iteration again, we derive
\begin{align*}
    \sup_X (\log\rho)
    &\leq 2^{\frac{n}{2}} \bt^{\frac{n(n-1)}{2}} 
    C_2^{\frac{n}{2}} \max\{\norm{\log\rho}_{L^2(X,\om)},1\}\\
    &\leq 2^{\frac{n}{2}} \bt^{\frac{n(n-1)}{2}} 
    C_2^{\frac{n}{2}} \max \lt\{(\sup_X \log\rho)^\oh \lt(\int_X \log\rho \om^n\rt)^\oh,1 \rt\}
\end{align*}
and thus
\[
    \sup_X (\log\rho)
    \leq C_3 \max\lt\{\int_X \log\rho \om^n, 1\rt\}
\]
for some constant $C_3 = C_3(n,B,C_S,V)$.

Now, everything comes down to bounding $\int_X \log \rho \om^n$ from above.
Using Poincaré inequality and (\ref{eq:g estimate 2}) with $p=1$, we get
\begin{equation}\label{eq:Gauduchon Poincare}
    \int_X \abs{\log \rho - \underline{\log\rho}}^2 \om^n
    \leq C_P \int_X \abs{\dd \log\rho}^2_\om \om^n
    \leq C_4 \int_X \log\rho \om^n
\end{equation}
where $\underline{\log\rho} = \frac{1}{V} \int_X \log\rho \om^n$ is the average of $\log\rho$ and $C_4 = C_4(n, B, C_P)$.
Then by (\ref{eq:dt-vol}), we can infer that
\begin{equation}\label{eq:Gaudhchon L1}
\begin{split}
    \dt \int_X \log\rho \om^n
    = V \dt \underline{\log\rho}
    &\leq V \int_{\{\rho < A\}} \underline{\log\rho} \om^n\\
    &\leq V \int_{\{\rho < A\}} \lt( \log A - \log\rho + \underline{\log\rho}\rt) \om^n\\
    &\leq V \int_X \lt( \abs{\log\rho - \underline{\log\rho}} + \log A\rt) \om^n.
\end{split}    
\end{equation}
We use (\ref{eq:Gauduchon Poincare}) and (\ref{eq:Gaudhchon L1}) to get
\begin{align*}
    \int_X \log\rho \om^n 
    &\leq \frac{V}{\dt} \lt( \int_X \abs{\log\rho - \underline{\log\rho}} \om^n + V \log A \rt)\\
    &\leq \frac{V}{\dt} \lt( V^\oh \lt(\int_X \abs{\log\rho - \underline{\log\rho}}^2 \om^n\rt)^\oh + V \log A \rt)\\
    &\leq \frac{V}{\dt} \lt(V^\oh C_4^\oh \lt(\int_X \log\rho \om^n\rt)^\oh + V \log A\rt).
\end{align*}
Note that if $x^2 \leq ax +b$ for $a, b >0$, then $x \leq \frac{a}{2}+(b+\frac{a^2}{4})^\oh$.
Eventually, we obtain
\[
    \int_X \log\rho \om^n
    \leq C_5(n,B,C_S,C_P,V)
\]
and this completes the proof of Theorem~\ref{thm:G bdd fix mfd}.

\section{Gauduchon current on the singular fibre}\label{sec:Gauduchon singular}
In this section, we construct a Gauduchon factor $\rho_0$ on $X_0^\reg$ as the limit of the Gauduchon factors on the nearby fibers $X_t$. 
In particular, we shall derive that the limit $\rho$ is bounded and thus $\rho^{\frac{1}{n-1}} \om_0$ is a bounded Gauduchon metric on $X_0^\reg$.
On the other hand, for a fixed irreducible, reduced, compact complex analytic space $X$, we shall show that the $(n-1)$-power of a bounded Gauduchon metric can be extend trivially to a pluriclosed current on whole $X$ and also prove a uniqueness result.

\smallskip

\subsection{Gauduchon metric on the central fibre}

\smallskip

\begin{thm}\label{thm:Gau current}
Suppose that $X_0$ is the central fibre in Setup~\ref{set:family setting}.
There exists a smooth function $1 \leq \rho \leq C_G$ on $X_0^\reg$ such that $\rho^\frac{1}{n-1} \om_0$ is a bounded Gauduchon metric. 
Here $C_G$ is the constant introduced in Corollary~\ref{cor:G bdd family}.
\end{thm}

\begin{proof}
We shall apply standard elliptic theory on some relatively compact subsets of $X_0^\reg$ to get a smooth function $\rho$.
This $\rho$ is the limit of $(\rho_{t_j})_{j \in \BN}$ defined on the fibre $X_{t_j}$ for some sequence $t_j \ra 0$ when $j \ra + \infty$.

Recall that $\pi: \CX \ra \BD$ is submersive on $X_0^\reg$.
By the tubular neighborhood theorem, there is an open neighborhood $\CM$ of $X^\reg_0$ in $\CX^\reg$ such that the following statements hold
\begin{enumerate}
    \item 
    For all $U \Subset X_0^\reg$, there exists an open set $\CM_U \subset \CX^\reg$, a constant $\dt_U > 0$, and a diffeomorphism $\psi_U: U \times \BD_{\dt_U} \xrightarrow{\sim} \CM_U$ such that the diagram
    \begin{equation*}
        \begin{tikzcd}
            U \times \BD_{\dt_U} \ar[r,"\psi_U","\mathrm{diffeo.}"']\ar[d,"\pr_2"]& \CM_U \subset \CM \subset \CX^\reg \ar[dl,"\pi"] \\
            \BD_{\dt_U}&
        \end{tikzcd}
    \end{equation*}
    commutes. 
    In particular, for all $t \in \BD_{\dt_U}$, $\psi_U(\cdot, t): U \ra \CM_U$ is a diffeomorphism onto its image $M_{U,t} = \psi_U(U,t) = \CM_U \cap X_t$.
    \item 
    If $U \Subset V \Subset X_0^\reg$, we have $\dt_U \geq \dt_V > 0$ and $\psi_U(x, t) = \psi_V(x, t)$ for all $x \in U$ and $t \in \BD_{\dt_V}$.
\end{enumerate}

Now, we denote by $P_t = \Dt_{\om_t}^\ast$ and fix $U_1 \Subset U_2 \Subset X_0^\reg$ which are connected open subsets.
Note that $U_i$ can be identified with $M_{U_i,t} = \psi_{U_2}(U_i,t)$ for $i \in \{1,2\}$ and for all $t \in \BD_{\dt_{U_2}}$ and hence $P_t$ can act on smooth functions defined on $U_2$.
Also, on $U_2$, the Riemannian metric $g_0$ induced by $\om_0$ is quasi-isometric to $g_t$ induced by $\om_t$, and the volume form $\om_0^n$ is comparable with $\om_t^n$.
In other words, we have a uniform constant $C_{U_2} > 0$ such that
\begin{equation}\label{eq:loc met bound}
    C_{U_2}^{-1} \iprod{\cdot}{\cdot}_{\om_0} \leq \iprod{\cdot}{\cdot}_{\om_t} \leq C_{U_2} \iprod{\cdot}{\cdot}_{\om_0},
    \text{ and }
    C_{U_2}^{-1} \om_0^n \leq \om_t^n \leq C_{U_2} \om_0^n.
\end{equation}

By G\r{a}rding inequality, we have
\[
    \norm{u}_{L^2_2(U_1,\om_0)} \leq C_{U_1, U_2}\lt( \norm{P_t u}_{L^2(U_2,\om_0)} + \norm{u}_{L^2(U_2,\om_0)} \rt)
\]
for all $u \in \CC^\infty_c(U_2)$.
The constant $C_{U_1,U_2}$ can be chosen independent of $t$ because the coefficients of $P_t$ move smoothly in $t$.
Choose a cutoff function $\chi$ such that $\supp(\chi) \subset U_2$ and $\chi \equiv 1$ on $U_1$.
From G\r{a}rding inequality, we obtain
\begin{equation}\label{eq:rho Garding 1}
    \norm{\rho_t}_{L^2_2(U_1,\om_0)} 
    \leq C_{U_1,U_2} \lt( \norm{P_t(\chi\rho_t)}_{L^2(U_2,\om_0)} + \norm{\chi\rho_t}_{L^2(U_2,\om_0)} \rt).
\end{equation}
In (\ref{eq:rho Garding 1}), the second term $\norm{\chi\rho_t}_{L^2(U_2,\om_0)}$ is uniformly bounded because of Corollary~\ref{cor:G bdd family}.
Hence, we only need to estimate $\norm{P_t(\chi\rho_t)}_{L^2(U_2,\om_0)}$.
Note that
\begin{equation}\label{eq:Pt cut eq}
\begin{split}
    P_t(\chi \rho_t) &= \frac{n}{\om_t^n} \ddc_t(\chi\rho_t \om_t^{n-1})\\
    &= \frac{n}{\om_t^n} \lt( 2\dd\rho_t\w\dc_t\chi\w\om_t^{n-1} 
    + \chi\ddc_t(\rho_t \om_t^{n-1}) + \rho_t \ddc(\chi\om_t^{n-1}) - \chi\rho_t\ddc_t\om_t^{n-1}\rt)\\
    &= 2 \iprod{\dd\rho_t}{\dd\chi}_{\om_t} + \rho_t P_t(\chi) - \rho_t\chi\frac{\ddc_t\om_t^n}{\om_t^n}.
\end{split}
\end{equation}
Obviously, $\rho_t P_t(\chi)$ and $\rho_t \chi \frac{\ddc_t\om_t^{n-1}}{\om_t^n}$ are uniformly bounded, so we only need to control the $L^2$-norm of the first term $\iprod{\dd\rho_t}{\dd\chi}_{\om_t}$:
\begin{equation*}
\begin{split}
    \int_{U_2} \abs{\iprod{\dd\rho_t}{\dd\chi}_{\om_t}}^2 \om_0^n
    &\leq C_{U_2}^2 \lt(\sup_{U_2} \abs{\dd\chi}^2_{\om_0}\rt) \int_{U_2} \abs{\dd\rho_t}_{\om_t}^2 \om_t^n\\
    &= C_{U_2}^2 \lt(\sup_{U_2} \abs{\dd\chi}^2_{\om_0}\rt) \int_{M_{U_2,t}} \abs{\dd\rho_t}_{\om_t}^2 \om_t^n\\
    &\leq C_{U_2}^2 \lt(\sup_{U_2} \abs{\dd\chi}^2_{\om_0}\rt) \int_{X_t} \abs{\dd\rho_t}^2_{\om_t} \om_t^n\\
    &\leq C_{U_2}^2 \lt(\sup_{U_2} \abs{\dd\chi}^2_{\om_0}\rt) \frac{nB}{2} \int_{X_t} \rho_t^2 \om_t^n.
\end{split}
\end{equation*}
Here the first line is by Cauchy--Schwarz inequality and (\ref{eq:loc met bound}).
The fourth line follows from an argument similar to the one used in (\ref{eq:g estimate 1}) (just replace $-p$ by $p$).
Since $1 \leq \rho_t \leq C_G$ and $\Vol_{\om_t}(X_t) \leq C_V$, we find a uniform bound of $\norm{P_t(\chi\rho_t)}_{L^2(U_2,\om_0)}$. 
Hence, $\norm{\rho_t}_{L^2_2(U_1,\om_0)}$ is uniformly bounded by some uniform constant $C(U_1,U_2)$.

For higher order estimates, we apply higher order G\r{a}rding inequalities on the fixed domains $U_1 \Subset U_2 \Subset X_0^\reg$:
\[
    \norm{u}_{L^2_{s+2}(U_1,\om_0)} \leq C_{s, U_1, U_2}\lt( \norm{P_t u}_{L^2_s(U_2,\om_0)} + \norm{u}_{L^2(U_2,\om_0)} \rt)
\]
for all $u \in \CC^\infty_c(U_2)$.
Let $\CU = (U_i)_{i\in\BN}$ be a relatively compact exhaustion of $X_0^\reg$.
Differentiating (\ref{eq:Pt cut eq}) on both sides and using a bootstrapping argument, we obtain $\norm{\rho_t}_{L^2_s(U_1,\om_0)} < C(s,\CU)$ where $C(s,\CU)$ does not depend on $t$.
By Rellich's theorem, there exists a subsequence $(\rho_{t_j})_{j \in \BN}$ such that $\rho_{t_j}$ converges to $\rho$ in $\CC^k(\overline{U_1})$ for all $k\in \BN$ when $t_j \ra 0$.
Therefore $\ddc_0(\rho \om_0^{n-1}) = \lim_{j \ra + \infty} \ddc_{t_j} (\rho_{t_j} \om_{t_j}^{n-1}) = 0$ on $U_1$.
Using a diagonal argument, we can infer that there is a smooth function $\rho$ on $X_0^\reg$ which is bounded between $1$ and $C_G$, and satisfies $\ddc_0(\rho \om_0^{n-1}) = 0$ on $X_0^\reg$. 
\end{proof}

\smallskip

\subsection{Proof of Theorem~\ref{bigthm:Gauduchon extention}}
In this subsection, we always assume that $X$ is an irreducible reduced compact complex space and $\om_\RG$ is a bounded Gauduchon metric on $X$. 
Before proving the uniqueness result and extension property, we recall 
the existence of cutoff functions with small $L^2$-gradients.
From a classical property in Riemannian geometry, these cutoff functions do exist on so-called \textit{parabolic} manifolds (see eg.~\cite{Glasner1983, EG1992}).
In our case, one can easily construct explicit cutoff functions by Hironaka's desingularization and $\log\log$-potentials (cf. \cite[Lemma 2.2]{Berndtsson2012} and \cite[Section 9]{Henri2013}):

\begin{lem}\label{lem:good cutoff}
Suppose that $(X,\om)$ is a compact hermitian variety. 
There exists cutoff functions $(\chi_\vep)_{\vep > 0} \subset \CC^\infty_c(X^\reg)$ satisfying the following properties:
\begin{enumerate}
    \item $\chi_\vep$ is increasing to the characteristic function of $X^\reg$ when $\vep$ decreases to zero;
    \item $\int_X \abs{\ddc\chi_\vep \w \om^{n-1}} \ra 0$ when $\vep \ra 0$; 
    \item $\int_X \dd\chi_\vep \w \dc\chi_\vep \w \om^{n-1} \ra 0$ when $\vep \ra 0$.
\end{enumerate}
\end{lem}

These functions allow us to perform integration by parts as on a compact manifold. 
Then we shall argue that the desired results hold when $\vep$ tends to zero.
Now, we give a proof of Theorem~\ref{bigthm:Gauduchon extention}:

\begin{proof}[Proof of Theorem~\ref{bigthm:Gauduchon extention}]
We divide the proof in two parts.

\smallskip

\noindent {\bf Part 1. Uniqueness of singular Gauduchon metrics:}
Assume that $\om_\RG$ and $\om_\RG'$ are both bounded Gauduchon metrics in the same conformal class.
We write $\rho$ to be the bounded Gauduchon factor satisfying $\rho^{\frac{1}{n-1}} \om_\RG = \om_\RG'$. 
Let $(\chi_\vep)_{\vep>0}$ be cutoff functions given in Lemma~\ref{lem:good cutoff}.
From Stokes formula and direct computations, we derive
\begin{equation}\label{eq:grad estimate}
    \int_{X^\reg} \chi_\vep |\dd \rho|^2_{\om_\RG} \om_\RG^n
    = \oh \int_{X^\reg} \rho^2 \ddc\chi_\vep \w \om_\RG^{n-1}.
\end{equation}
Indeed, 
{\small
\begin{align*}
    \int_{X^\reg} \chi_\vep |\dd \rho|_{\om_\RG}^2 \om_\RG^{n-1}
    &= \int_{X^\reg} \chi_\vep \dd \rho \w \dc \rho \w \om_\RG^{n-1}
    = - \int_{X^\reg} \rho \dd \lt( \chi_\vep \dc \rho \w \om_\RG^{n-1} \rt)\\
    &= - \int_{X^\reg} \rho \lt( \dd \chi_\vep \w \dc \rho \w \om_\RG^{n-1} + \chi_\vep \ddc \rho \w \om_\RG^{n-1} - \chi_\vep \dc \rho \w \dd \om_\RG^{n-1}\rt)\\
    &= -\oh \int_{X^\reg} \dd \rho^2 \w \dc \chi_\vep \w \om_\RG^{n-1} + \int_{X^\reg} \rho \chi_\vep \dd \rho \w \dc \om_\RG^{n-1}\\
    &= \oh \int_{X^\reg} \rho^2 \dd \lt( \dc \chi_\vep \w \om_\RG^{n-1}\rt)
    - \oh \int_{X^\reg} \rho^2 \dd\lt(\chi_\vep \dc\om_\RG^{n-1}\rt)\\
    &= \oh \int_{X^\reg} \rho^2 \ddc \chi_\vep \w \om_\RG^{n-1}
    - \oh \int_{X^\reg} \rho^2 \dc \chi_\vep \w \dd \om_\RG^{n-1}
    - \oh \int_{X^\reg} \rho^2 \dd \chi_\vep \w \dc \om_\RG^{n-1}\\
    &= \oh \int_{X^\reg} \rho^2 \ddc \chi_\vep \w \om_\RG^{n-1}.
\end{align*}
}
By Lemma~\ref{lem:good cutoff} and the boundedness of $\om_\RG$, we can see that $\int_{X} |\ddc \chi_\vep \w \om_G^{n-1}|$ converges to zero when $\vep$ tends to zero. 
Since $\rho$ is bounded, the RHS of (\ref{eq:grad estimate}) goes to zero and hence 
\[
    \int_{X^\reg} |\dd \rho|^2_{\om_\RG} \om_\RG^n = 0.
\]
Because $X^\reg$ is connected by the irreducible assumption, $\rho$ is a constant.

\smallskip

\noindent {\bf Part 2. Extension to a pluriclosed current:} 
We already have a smooth pluriclosed bounded positive $(n-1,n-1)$-form $\om_\RG^{n-1}$ on $X^\reg$. 
Note that $\om_\RG^{n-1}$ can extend trivially as a bounded positive $(n-1,n-1)$-current $T$ on $X$.
If $\om_G^{n-1}$ is closed, so would be $T$ by Skoda--El Mir extension result \cite{Skoda1982, ElMir1984}.
Some results for plurisubharmonic currents do exist in the literature (see eg. \cite{AB1993, DEE2003}), but we could not find the one of interest for us here. 
Therefore, we provide a quick proof.

Again, we are going to use good cutoff functions $(\chi_\vep)_{\vep > 0}$ in Lemma~\ref{lem:good cutoff} to prove that $T$ is pluriclosed in the sense of currents.
Fix a smooth function $f$ on $X$.
We need to show that 
\[
    \iprod{f}{\ddc T} := \int_{X} \ddc f \w T = 0.
\]
Using the cutoff functions constructed in Lemma~\ref{lem:good cutoff}, we can write
\[
    \int_{X} \ddc f \w T 
    = \underbrace{\int_{R_\vep} \chi_\vep \ddc f \w T}_{:= \RN{1}_\vep} 
    + \underbrace{\int_{S_\vep} (1-\chi_\vep) \ddc f \w T}_{:=\RN{2}_\vep}
\]
where $R_\vep$ is a small open neighborhood of the closure of $\{\chi_\vep > 0\}$ contained in $X^\reg$ and similarly $S_\vep$ is a small neighborhood of the closure of $\{\chi_\vep < 1\}$.
According to Lemma~\ref{lem:good cutoff}, as $\vep \ra 0$, $R_\vep$ tends to $X^\reg$ and $S_\vep$ shrinks to $X^\sing$.
Therefore, $\RN{2}_\vep$ converges to zero when $\vep$ goes to zero.
We compute the term $\RN{1}_\vep$ by Stokes formula:
\begin{align*}
	\RN{1}_\vep
	&= \int_{R_\vep} f \ddc\lt( \chi_\vep \om_\RG^{n-1}\rt)
	= \int_{R_\vep} f \lt(\ddc \chi_\vep \w \om_\RG^{n-1} + 2 \dd \om_\RG^{n-1} \w \dc \chi_\vep\rt)\\
	&= -\underbrace{\int_{R_\vep} f \ddc \chi_\vep \w \om_\RG^{n-1}}_{:= \RN{3}_\vep} 
	- 2 \underbrace{\int_{R_\vep} \dd f \w \dc \chi_\vep \w \om_\RG^{n-1}}_{:= \RN{4}_\vep}.
\end{align*}
By Cauchy--Schwarz inequality, the term $\RN{4}_\vep$ can be bounded by 
\[
    |\RN{4}_\vep| 
    \leq 
    \underbrace{\int_{R_\vep} \1_{\{\dd \chi_\vep \neq 0\}} \dd f \w \dc f \w \om_\RG^{n-1}}_{:= \RN{5}_\vep}
    + \underbrace{\int_{R_\vep} \dd\chi_\vep \w \dc\chi_\vep \w \om_\RG^{n-1}}_{:= \RN{6}_\vep}.
\]
Using the dominated convergence theorem, $\RN{5}_\vep$ converges to zero when $\vep$ goes to zero, because the set $\{\dd\chi_\vep \neq 0\}$ is contained in $S_\vep$ and $S_\vep$ shrinks to $X^\sing$. 
Applying Lemma~\ref{lem:good cutoff}, $\RN{3}_\vep$ and $\RN{6}_\vep$ converge to zero as $\vep$ tending to zero.
These yield that $\RN{1}_\vep \ra 0$ as $\vep \ra 0$.
All in all, we have $\int_{X} \ddc f \w T = 0$ for all $f \in \CC^\infty(X)$; hence $\ddc T = 0$ in the sense of currents.
\end{proof}

\bibliographystyle{smfalpha}
\bibliography{biblio_SGM}

\end{document}